\documentclass[11pt]{amsart}
\usepackage{etoolbox}
\makeatletter
\patchcmd{\@startsection}{\@afterindenttrue}{\@afterindentfalse}{}{}
\makeatother
\usepackage[T1]{fontenc}
\usepackage[utf8]{inputenc}
\usepackage[margin=1in,a4paper]{geometry}
\usepackage{graphicx} % Required for inserting images
\usepackage[english]{babel}
\usepackage{amsmath, amssymb, amsthm}
\usepackage{url}
\usepackage{subfigure}

\usepackage[dvipsnames]{xcolor}

\newcommand\jk[1]{\textcolor{ForestGreen}{#1}}

\theoremstyle{plain}
\newtheorem{theorem}{Theorem}[section]

\newtheorem{lemma}[theorem]{Lemma}
\newtheorem{conjecture}[theorem]{Conjecture}
\newtheorem*{conjecture*}{Conjecture}
\newtheorem{corollary}[theorem]{Corollary}

\theoremstyle{definition}

\theoremstyle{remark}

\title[On $4$-perfect-matching covers\ldots]{On $4$-perfect-matching covers of cubic graphs with\\two adjacent odd circuits in a 2-factor}

\author[J.~Karab\'a\v{s}]{J\'an Karab\'a\v{s}}
\author[E.~M\'a\v{c}ajov\'a]{Edita M\'a\v{c}ajov\'a}
\address[J.~Karab\'a\v{s}]{FEEIT, Slovak University of Technology, Ilkovi\v{c}ova 3, 84104 Bratislava, Slovakia}
\address[J.~Karab\'a\v{s}]{Mathematical Institute of Slovak Academy of Sciences, \v{D}umbierska 1, 97411 Bansk\'a Bystrica, Slovakia}
\address[E.~M\'a\v{c}ajov\'a]{Faculty of Mathematics, Physics and Informatics, Comenius University, Mlynsk\'a dolina, 84248 Bratislava, Slovakia}
\email[J.~Karab\'a\v{s}]{jan.karabas@stuba.sk}
\email[E.~M\'a\v{c}ajov\'a]{macajova@dcs.fmph.uniba.sk}

\begin{document}
\begin{abstract}
	Let $G$ be a cubic graph admitting a $2$-factor consisting of exactly two odd circuits, and let the complementary $1$-factor contain precisely three spokes (along with an arbitrary number of chords). We show that four perfect matchings can cover $G$. As a consequence, $G$ fulfils the $7/5$-Conjecture of Alon and Tarsi.
\end{abstract}

\maketitle
\section{Introduction}
Perfect matchings constitute a fundamental substructure of cubic graphs. When a cubic graph is $3$-edge-colourable, each colour class in every proper $3$-edge-colouring corresponds to a perfect matching. By 
Sch\"{o}nberger's theorem~\cite{LP86, Sch34}, each edge of a bridgeless cubic graph belongs to at least one perfect matching. The minimum number of perfect matchings required to cover the edge set of a graph $G$ is called the perfect matching index of $G$. Berge in the 1970s conjectured that, for every bridgeless cubic graph, five perfect matchings suffice to cover all its edges~\cite{M14}. To date, this conjecture remains widely open and has been proved only for very specific families of cubic graphs.

In the last fifteen years, much attention has been paid to finding cubic graphs that do not admit a covering of their edge set by four perfect matchings. Deciding whether, for a given cubic graph that does not admit a $3$-edge-colouring, a covering of its edge set by four perfect matchings exists is NP-complete~\cite{EM14, SV24}. Identifying nontrivial cubic graphs (that is, in this context, cubic graphs with cyclic connectivity at least 4 and girth at least 5) that require at least five perfect matchings to cover their edges is thus extremely difficult. Computer search~\cite{BGHM13, BO26} reveals that among nontrivial snarks of order at most $38$, there are only two such graphs: namely, the Petersen graph and a graph on $34$ vertices, the so-called `windmill snark'. The discovery of the windmill snark served as a foundation for constructing infinite families of graphs, each of whose members has perfect matching index at least $5$ \cite{AKLM16, EM14, MS21}.

The class of cubic graphs with perfect matching index at least $5$ is of particular interest, as potential counterexamples to well-known conjectures, such as the 5-cycle double cover conjecture~\cite{S73}, the 7/5-Conjecture of Alon and Tarsi~\cite{AT85}, and the Fan-Raspaud conjecture~\cite{FR94}, must come from it~\cite{S15}. On the other hand, researchers are trying to identify families of graphs as large as possible with perfect matching index at most $4$, see e.g.~\cite{FV08, KMNS-pmi}. Since it is known \cite{S15} that if the perfect matching index of a cubic graph $G$ is at most 4, then it has a cycle cover of total length at most $4/3\cdot |E(G)|$, all these graphs fulfil the $7/5$-Conjecture of Alon and Tarsi and also the other two conjectures mentioned above. 

A cubic graph is a \emph{strong snark} if the removal of any edge yields a graph homeomorphic to a cubic graph admitting no proper $3$-edge-colouring. The following conjecture concerning strong snarks was proposed in~\cite{KMNS-berggen}. 

\begin{conjecture*}[\cite{KMNS-berggen}]
	Let $G$ be a nontrivial cubic graph with perfect matching index greater than~$4$, different from the Petersen graph. Then $G$ is a strong snark.
	%Let $G$ be a nontrivial strong snark. Then either the perfect matching index of $G$ is at most 4 or $G$ is the Petersen graph.
\end{conjecture*}

Here we take a step towards solving this conjecture. Recall that if a graph has a $2$-factor with two odd circuits (and arbitrarily many even circuits), with two adjacent vertices in different odd circuits of the $2$-factor, then the graph is not strong. In the current paper, we investigate the family of cubic graphs having a $2$-factor consisting of exactly two odd
circuits, such that the complementary $1$-factor contains precisely three edges between the circuits (and arbitrarily many edges of the $1$-factor connect vertices inside the two circuits), and prove that such a graph can be covered with four perfect matchings. %Thus, removing a single edge between the odd circuits does not yield a $3$-edge-colourable graph.
Our proof employs a nonstandard two-phase induction.

%\jk{Oddness of a cubic graph is the smallest number of odd circuits in a $2$-factor of the graph. The oddness of any bridgeless cubic graph is an even number, and a cubic graph has oddness $0$ if and only if it is $3$-edge-colourable. (tato veta sa mi zda zbytocna, oddness v zasade nepotrebujeme)}

\section{Preliminaries}
A graph $G=(V,E,\xi)$ is given by finite sets $V=V(G)$ of \emph{vertices} and $E = E(G)$ of \emph{edges}, together with the incidence map $\xi\colon E \to \binom{V}{2}\cup\binom{V}{1}$ assigning one or two \emph{end-vertices} to every edge $e\in E$. An edge with two end-vertices assigned is a (regular) edge, while the one with a unique end-vertex assigned is a \emph{dangling edge}. We say that a vertex $v\in V$ is \emph{incident with an edge} $e\in E$ if $v\in\xi(e)$. The number of edges incident with a vertex $v$ is called the \emph{degree of $v$}, denoted by $\operatorname{deg}(v)$. A graph in which all vertices are of degree $k$ is called a \emph{$k$-regular graph}; $3$-regular graphs, or shortly \emph{cubic graphs}, are of interest in this paper. For the definitions of other concepts not explicitly defined in this text, the reader is referred to~\cite{BM08-book,diestel17-book}.

A subgraph $S=(V', E',\xi')$ of a graph $G=(V, E,\xi)$ is a graph such that $V'\subseteq V$, $E'\subseteq E$, and $\xi'=\xi|_{E'}$; we write $S\subseteq G$. A $1$-regular subgraph $M\subseteq G$ of a graph $G$ is called a \emph{matching}. A matching $M\subseteq G$ is \emph{maximum} if it has the maximum number of edges among all matchings of $G$. A matching spanning the whole $V$ is called a \emph{perfect matching} in the graph $G$. Let $P$ be a path. The symbol $|P|$ denotes the number of edges of $P$. The path $P$ is even (odd) if $|P|$ is even (odd). A $2$-regular subgraph of a graph is called a \emph{cycle}; a connected cycle will be called a \emph{circuit}. It is worth mentioning that the complement $G-M$ of any perfect matching $M$ in a cubic graph $G$ is a cycle. Moreover, if $M$ and $N$ are matchings in a graph, then each component of $M\cup N$ is either a circuit or a path.

A cubic graph with at least one dangling edge is called a \emph{multipole}; specifically, the one with $k$ dangling edges is a \emph{$k$-pole}. A $3$-pole is called \emph{Hamiltonian} if it contains a circuit traversing all its vertices. We will use the notation $(G, H)$ for a Hamiltonian cubic $3$-pole $G$ with a distinguished Hamiltonian circuit $H$. A dangling edge of a Hamiltonian $3$-pole $(G, H)$ is called a \emph{spoke}. A regular edge of $(G, H)$ is called a \emph{chord} if it does not belong to $H$. The set of spokes will be denoted by $K$ and the set of chords will be denoted by $Q$; it follows that $E(G)=E(H)\cup Q\cup K$. In this paper, we always consider exactly three spokes,  $K = \{e_1, e_2, e_3\}$; the end-vertex of $e_i$ is denoted by $v_i$. A path $S\subseteq H$ whose two end-vertices are the end-vertices of spokes and which does not contain the end-vertex of the third spoke will be called a \emph{segment}. If $S$ is a segment, then an \emph{exterior chord of $S$} is a chord that has exactly one end-vertex in $S$, and this vertex is an inner vertex of $S$. An \emph{alternating circuit} in $(G, H)$ is a circuit alternating edges of $H$ and $Q$. 

Let $\mathcal{M}=\{M_1, M_2, \ldots, M_r\}$ be a family of perfect matchings of a graph $G$. We say that $\mathcal{M}$ \emph{covers} $G$ if, for every edge $e\in E(G)$, there is a perfect matching $M_i\in\mathcal{M}$ such that $e\in E(M_i)$. In this case, $\mathcal{M}$ is an \emph{$r$-perfect-matching cover} of $G$. In this paper, we shall exclusively deal with $4$-perfect-matching covers of cubic multipoles, that is, $\mathcal{M}=\{M_1, M_2, M_3, M_4\}$. For convenience, we refer to a $4$-perfect-matching cover simply as a \emph{$4$-cover} of a cubic graph. Given a $4$-cover of a cubic graph (multipole), every vertex is incident with two edges covered by one perfect matching, and with one edge covered by two perfect matchings. The edges covered by a single perfect matching are called \emph{simply covered}, while those covered by two perfect matchings are \emph{doubly covered}. Let $(G, H)$ be a Hamiltonian cubic $3$-pole. Four perfect matchings $\{M_1, M_2, M_3, M_4\}$ form a \emph{proper $4$-perfect-matching cover}, or briefly, a \emph{proper $4$-cover}, of $(G, H)$ if $M_4=Q\cup K$ and all three spokes are doubly covered. Parity arguments imply that one of the spokes is in $M_1\cap M_4$, another is in $M_2\cap M_4$, and the last one is in $M_3\cap M_4$. 
% Parity arguments imply that the perfect matching is different from $M_4$ is different from each of $e_1,e_2,e_3$, thus showing that there exists the required cover where $e_i\in M_4\cap M_i$ for each $i\in\{1,2,3\}$.

\section{Main Theorem}

The main result of the paper reads as follows.

\begin{theorem}\label{thm:main}
A Hamiltonian cubic $3$-pole $(G, H)$ has a proper $4$-perfect-matching cover. 
\end{theorem}
\noindent{}First, we have to discuss several essential facts.

\smallskip
Given a Hamiltonian $3$-pole $(G, H)$, an alternating circuit $C$ (alternating edges from $Q$ and from $H$) can be \emph{suppressed} in $(G, H)$ such that the resulting $3$-pole $(G', H')$ is again a Hamiltonian $3$-pole. The suppression of $C$ is done by removing the chords $C\cap Q$ of $(G, H)$ and by smoothing the resulting vertices of degree $2$ on the circuit $H$. The chords of $(G,H)$ that do not belong to $C$ and the spokes persist in $(G',H')$. The following property of proper $4$-covers is important. 

\begin{lemma} \label{lem:alt_c} Let $C$ be an alternating circuit in $(G, H)$. Let $(G', H')$ be the Hamiltonian $3$-pole created by suppression of $C$. If $(G', H')$ has proper $4$-cover, then so does $(G, H)$.
\end{lemma}

\begin{proof}
We extend a proper $4$-cover of $(G',H')$ to a proper $4$-cover of $(G, H)$. All edges of $C\cap Q$ will be simply covered with $M_4$. Each edge $e$ in $H'$ corresponds to an odd path $P_e$ in $H$ where the edges alternate between those not in $C$ and those in $C$. First, assume that $e$ is simply covered in $H'$, say with a perfect matching $M_a$. Since all chords and spokes in $(G', H')$ are covered with $M_4$ (and possibly one other perfect matching), we have $a\in\{1,2,3\}$. On the path $P_e$ in $H$ we alternate the edge simply covered by $M_a$ with doubly covered edges, covered by $M_b$ and $M_c$ (shortly $M_{bc}$) where $\{a,b,c\}=\{1,2,3\}$. The edges incident with the starting and terminal vertices of $P_e$ will be simply covered with $M_{a}$, and each internal vertex of $P_e$ has three incident edges covered with $M_4$, $M_a$, and $M_{bc}$, respectively. Similarly, if $e$ is doubly covered in $H'$, say by $M_a$ and $M_b$, the path $P_e$ will consist of doubly covered edges, covered by $M_{ab}$ and simply covered edges, covered by $M_c$, and the edges incident with the starting and terminal vertices of $P_e$ will be doubly covered with $M_{ab}$. As a result, a proper $4$-cover of $(G, H)$ is constructed.
\end{proof}

\noindent{}Let $(G, H)$ be a Hamiltonian cubic 3-pole; therefore, it has an odd number of vertices. The vertices $v_1, v_2, v_3$ either split $H$ into three segments of odd length or into two segments of even length and one segment of odd length. Briefly, we will use even and odd segments, respectively. The situation when there are three odd segments in $(G, H)$ is treated separately.

\begin{lemma}\label{lem:allodd}
A Hamiltonian $3$-pole $(G, H)$ with three odd segments has a proper $4$-perfect-matching cover.
\end{lemma}
\begin{proof}
We define a proper $4$-cover of $(G, H)$ as follows. The chords and spokes of $(G, H)$ will constitute the perfect matching $M_4$. The spoke $e_i$ will also belong to a perfect matching $M_i$. The segment $S_{ab}$ between $v_a$ and $v_b$ alternates $M_c$ and $M_{ab}$, respectively, for each segment $S_{ab}$ where $\{a,b,c\} = \{1,2,3\}$.
\end{proof}

In the following we assume that $(G,H)$ has two even segments and a unique odd segment, unless stated otherwise. The vertices $v_1, v_2, v_3$ of $H$ are the boundary vertices determining the segments. The even segments $E_1$ and $E_2$ are the paths with end-vertices $v_1$, $v_3$ and $v_2$, $v_3$, respectively. The odd segment $O$ is the path with end-vertices $v_1$ and $v_2$.

\begin{lemma}\label{lem:segm2}
A Hamiltonian $3$-pole $(G, H)$ with a segment of length $2$ is $3$-edge-colourable. Moreover, $(G, H)$ has a proper $4$-perfect-matching cover.
\end{lemma}

\begin{proof}
Without loss of generality we may assume that $|E_1|=2$. Attach the three free ends of spokes $e_1, e_2$, and $e_3$ to a new vertex $x$; the resulting graph will be denoted by $G^*$. We will show that $G^*$ is 3-edge-colourable. 

To see this, we create an auxiliary graph $B$. Let $y$ be the vertex that is adjacent to both $v_1$ and $v_3$; thus, $G^*$ contains a $4$-circuit $v_1yv_3x$. Let $v_1'$ be the neighbour of $v_1$ different from $x$ and $y$, let $v_3'$ be the neighbour of $v_3$ different from $x$ and $y$ and let $y'$ be the vertex such that $yy'$ is a chord in $G$. To create $B$, delete the vertices $v_1yv_3x$, leaving four new dangling edges incident with $v'_1, v'_3, v_2$ and $y'$ respectively. Reconnect $v_1'$ and $v'_3$ by removing the two dangling edges incident to $v'_1$ and $v'_3$, respectively, and by adding a new edge $f_1$ such that $\xi(f_1) = \{v'_1, v'_3\}$. Reconnect the vertices $y'$ and $v_2$ in the same fashion. The resulting graph $B$ is derived from $G^*$ by the removal of the vertices of the $4$-cycle $v_1yv_3x$ and by restoring $3$-regularity. It is Hamiltonian, thus $3$-edge-colourable. It is a widely recognised fact that if $B$ is $3$-edge-colourable, then so is $G^*$. A $3$-edge-colouring of $G^*$ induces a $3$-edge-colouring of $G$.

Since $G^*$ is $3$-edge-colourable, we can assume that spoke $e_i$ is coloured with the colour $i\in\{1,2,3\}$. Each matching $M_i$ then corresponds to the colour class indexed by $i$. The perfect matching $M_4$ corresponds to the set of chords and spokes of $(G, H)$. The collection of perfect matchings $\mathcal{M} = \{M_1, M_2, M_3, M_4\}$ forms a proper 4-cover of~$(G, H)$. 
\end{proof}

\begin{lemma}\label{lem:uniq}
A Hamiltonian $3$-pole $(G, H)$ with a segment possessing a unique exterior chord has a proper $4$-perfect-matching cover.
\end{lemma}

\begin{proof}
Let $(G, H)$ be a $3$-pole with a segment $S$ with a unique exterior chord. It is easy to see that $S$ is even. Without loss of generality, we may assume that $S=E_1$, between vertices $v_1$ and $v_3$, has a unique exterior chord incident with a vertex $u$ outside of $E_1$. We create a Hamiltonian $3$-pole $(G', H')$ from $(G, H)$ by deleting all inner vertices of $E_1$ and by introducing a new vertex $q$, connected to $v_1$, $v_3$, and $u$. The Hamiltonian circuit $H'$ in $G'$ inherits the segments $E_2$ and $O$ from $H$ and gets two new edges of the path $v_1qv_3$. The 3-pole $G'$  has a segment $E'_1$ of length $2$; by Lemma~\ref{lem:segm2}, it has a proper $4$-cover. 

We extend the covering of $G'$ to $G$. The edge $qv_1$ of $G'$ is simply covered by $M_a$, for some $a\in\{1,2,3\}$, since it is adjacent to a doubly covered edge $e_1$. For the same reasons, the edge $qv_3$ is simply covered by $M_b$, $b\in\{1,2,3\}$, $a\neq b$. Again, the parity arguments imply that the only exterior chord of $E'_1$ is covered with $M_c$ and $M_4$, where $\{a,b,c\}=\{1,2,3\}$.  We define the covering of $G$ in such a way that the edges not incident with the inner vertices of $E_1$ will be covered with the same perfect matching(s) as they were in $G'$. The edges on $E_1$ will be alternately covered with $M_a$ and $M_b$ when proceeding from $v_1$ to $v_3$, and finally, all the chords with at least one end-vertex in $E_1$ will be covered with $M_c$ and $M_4$. In this way, all the edges of $G$ are covered, and the perfect matching $M_4$ covers chords and spokes only, as required.   
\end{proof}

All Hamiltonian $3$-poles that satisfy the assumptions of Lemma~\ref{lem:allodd}, Lemma~\ref{lem:segm2}, or Lemma~\ref{lem:uniq} possess a proper $4$-cover. It remains to address Hamiltonian $3$-poles $(G, H)$ with two even segments where the two even segments $E_1$ and $E_2$ have more than one exterior chord. The remaining segment $O$ of the Hamiltonian circuit is odd. We denote this family of Hamiltonian cubic $3$-poles by $\mathcal{G}$. The parity argument implies that each $E_i$ has at least three exterior chords.  By induction on $|E_1|+|E_2|$, we will demonstrate that all graphs in $\mathcal{G}$ possess an alternating circuit. Consequently, by applying Lemma~\ref{lem:alt_c}, no member of $\mathcal{G}$ can be the smallest counterexample for Theorem~\ref{thm:main}.

\begin{lemma}\label{lem:induction}
	All Hamiltonian $3$-poles in $\mathcal{G}$ have an alternating circuit.
\end{lemma}
 \begin{proof} Since each of $E_1$ and $E_2$ has at least three exterior chords, we have $|E_i|\ge 4$. The basis of induction is thus formed by $3$-poles $G$ with $|E_1|=|E_2|=4$ and with exactly three exterior chords in each $E_i$.

 \medskip
 \noindent\emph{Base case.} Let $(G, H)$ be a Hamiltonian $3$-pole whose odd segment $O$ has end-vertices $v_1$ and $v_2$. Let $M_O$ be the maximum matching in $O-\{v_1, v_2\}$. Let the three vertices in $E_1$ be denoted $a_1, a_2, a_3$ when proceeding from $v_1$ to $v_3$ and let the three vertices in $E_2$ be denoted $b_1, b_2, b_3$ when proceeding from $v_2$ to $v_3$. 

  The components of $Q\cup M_O$ are even circuits and paths that start and end at the vertices $a_1, a_2, a_3, b_1, b_2, b_3$; we have three such paths in $Q\cup M_O$.  If there is a circuit in $Q\cup M_O$, it is an alternating circuit, and we are done. We therefore assume that $Q\cup M_O$ is formed by exactly three paths with the end-vertices $a_1, a_2, a_3, b_1, b_2, b_3$. 

  We start with the case that each of these paths has one vertex in $\{a_1, a_2, a_3\}$ and the other in $\{b_1, b_2, b_3\}$. We denote the paths starting at $a_i$ by $P_i$. Assume first that the terminal vertex of $P_2$ is $b_2$. Let $b_k$ be the terminal vertex of $P_1$, where $k\in\{1,3\}$. Then $P_1b_kb_2P_2a_2a_1$ is an alternating circuit. Assume now that $P_2$ ends at $b_m$ for $m\in\{1,3\}$. Let $P_l$ be the path that ends at $b_2$; clearly, $l\in\{1,3\}$. Then $P_lb_2b_mP_2a_2a_l$ is an alternating circuit. The remaining case is that there is one path with both ends in $\{a_1,a_2,a_3\}$, one path with both ends in $\{b_1,b_2,b_3\}$ and one path with one end in $\{a_1,a_2,a_3\}$ and the other in $\{b_1,b_2,b_3\}$. Let $R$ be the path that has one end-vertex $a_2$. If the other end of $R$ is $a_1$ or $a_3$, say $a_1$, then $Ra_1a_2$ is an alternating circuit; similarly for the path with end-vertex at $b_2$. It remains to consider the case that one of the paths has the end-vertices $a_1$ and $a_3$, the other $b_1$ and $b_3$, and the third path has $a_2$ and $b_2$ as its end-vertices.

  Let the path between $a_1$ and $a_3$ be called $R_a$, let the path between $b_1$ and $b_3$ be called $R_b$ and let $R$ be the path between $a_2$ and $ b_2$. The edges of $R_a$ will be called \emph{red}, the edges of $R_b$ will be called \emph{green}, and the edges of $R$ will be called \emph{grey}. As there are no circuits in $Q\cup M_O$, each edge of $M_O$ is either red, green or grey.  Since each $E_i$ has three exterior chords and one of these chords has the other end in $E_{3-i}$, both $R_a$ and $R_b$ intersect the segment $O$. Let $W$ be a shortest subpath of $O$ which has one end-vertex incident with a green edge and the other end-vertex with a red edge; this implies that uncoloured and grey edges alternate on $W$. The end-vertex of $W$ incident with the red edge is denoted by $z_r$, and the one incident with the green edge is $z_g$. Let $z_r'$ be the vertex on $O$ outside $W$ adjacent to $z_r$ (the edge $z_rz_r'$ is red) and let $z_g'$ be the vertex on $O$ outside $W$ adjacent to $z_g$ (the edge $z_gz_g'$ is green). Let $R_a'$ be the subpath of $R_a$ between the vertices $z_r'$ and $a_i$ for some $i\in\{1,3\}$ not containing $z_r$ and let $R_b'$ be the subpath of $R_b$ between the vertices $z_g'$ and $b_j$ for some $j\in\{1,3\}$ not containing $z_g$. Let $N_O$ be the maximum matching of $O-\{v_1,v_2,z_r',z_g'\}$ and let $N=N_O\cup \{a_2a_{3-i},b_2b_{3-j}\}$. The edges of $(N\cup Q)-(E(R_a')\cup E(R_b')\cup \{e_1,e_2,e_3\})$ induce a $2$-regular graph that alternates edges of $H$ and $Q$. Moreover, this $2$-regular graph is non-empty as it contains edges $a_2a_{3-i}$ and $b_2b_{3-j}$; therefore it must contain an alternating circuit.

  \medskip
  We now proceed to the induction step.
  We assume that all $3$-poles in $\mathcal{G}$ such that the sum of the lengths of the even segments is smaller than that of $(G, H)$ have an alternating circuit. Further, we assume that the length of at least one of $E_1$ and $E_2$ is greater than~$4$. 

 \smallskip
 \noindent\emph{Induction step.} Without loss of generality we will assume that $|E_1|>4$ and let $v_1, u_1, u_2,\ldots, u_h, v_3$ be the vertices in $E_1$, ordered in the direction from $v_1$ to $v_3$. Recall that $h$ is odd. Let $u_1'$ and $u_2'$ be the vertices in $H$ such that $u_1u_1'$ and $u_2u_2'$ are chords. Create a new Hamiltonian $3$-pole $(G', H')$ \jk{by} removing the vertices $u_1$ and $u_2$ and their incident edges and connecting $v_1$ with $u_3$ and $u_1'$ with $u_2'$ with an edge. The Hamiltonian circuit $H'$ of $G'$ will be formed by $v_1u_3u_4\ldots u_hv_3E_2O$ and the three segments of $H'$ will be $E_1'$, $E_2$, and $O$ where $E_1'=v_1u_3u_4\ldots u_hv_3$ is even. Assume that each of $E_1'$ and $E_2$ had at least three exterior chords. By the induction hypothesis, $(G', H')$ contains an alternating circuit, say $A'$. The circuit $A'$ can be extended to an alternating circuit $A$ of $H$ such that if $A'$ contains the edge $u_1'u_2'$, we replace it by the path $u_1'u_1u_2u_2'$, and if $A'$ does not contain $u_1'u_2'$, then $A$ corresponds exactly to $A'$. Since the edge $v_1u_3$ could not belong to $A'$ because $e_1$ is a spoke, $A$ is indeed an alternating circuit. Assume therefore that at least one of $E_1'$ and $E_2$ does not contain three exterior chords in $(G', H')$. This implies that both $u_1u_1'$ and $u_2u_2'$ are exterior chords for $E_1$. It follows that $u_1'$ and $u_2'$ belong to $E_2\cup O$. 

 Apply the same procedure to $u_h$, $u_{h-1}$ in $H$ as we did to $u_1$, $u_2$, and call the $3$-pole created this way $(G'', H'')$. The edges $u_hu_h'$ and $u_{h-1}u_{h-1}'$ are the exterior chords of $E_1''$, where $E_1''$ is the corresponding even segment, and since the number of exterior chords of even segments is odd, we have that $E_1''$ has at least three exterior chords and $|E_1''|\ge 4$. Assume that at most one of $u_h'$ and $u_{h-1}'$ belongs to $E_2$. Then the number of exterior chords in $E_2$ remains the same as in $(G, H)$; therefore, $(G'', H'')$ belongs to $\mathcal{G}$ and we can use the induction hypothesis on $(G'', H'')$, which has an alternating circuit, and therefore so does $(G, H)$. It remains to consider the case when both $u_h'$ and $u_{h-1}'$ belong to $E_2$, but in this case, in $(G', H')$, both even segments have at least two (and by parity argument, at least three) exterior chords, and this case is already covered. This finishes the induction step.
\end{proof}

\begin{proof}[Proof of Theorem~\ref{thm:main}]
	Let $(G, H)$ be a smallest counterexample to Theorem~\ref{thm:main} with respect to the number of vertices. By Lemmas~\ref{lem:allodd}, \ref{lem:segm2}, and~\ref{lem:uniq}, it follows that $(G, H) \in \mathcal{G}$. Further, by Lemma~\ref{lem:induction}, the Hamiltonian $3$-pole $(G, H)$ contains an alternating circuit. Applying Lemma~\ref{lem:alt_c}, we deduce that $(G, H)$ possesses a proper $4$-cover, as $(G', H')$ obtained from $(G, H)$ by the suppression of the alternating circuit is smaller and must therefore have a proper $4$-cover. Consequently, such a counterexample cannot exist.
\end{proof}

\begin{corollary} Let $G$ be a cubic graph with a $2$-factor consisting of two odd circuits, whose complementary $1$-factor contains exactly three spokes (and any number of chords). Then $G$ can be covered by four perfect matchings.
\end{corollary}
\begin{proof}
	We split $G$ into two 3-poles; for each, apply Theorem~\ref{thm:main}, and combine the four perfect matchings on both 3-poles to form four perfect matchings covering all the edges of $G$.
\end{proof}

In 2015, it was shown in \cite{S15} that if the perfect matching index of a cubic graph is at most 4, then it has a cycle cover of total length $4/3\cdot|E(G)|$.

\begin{corollary}
Let $G$ be a nontrivial cubic graph with a $2$-factor consisting of two odd circuits, whose complementary $1$-factor contains exactly three spokes (and any number of chords). Then $G$ has a cycle cover of total length $4/3\cdot|E(G)|$, and therefore it fulfils the $7/5$-Conjecture of Alon and Tarsi.
\end{corollary}

We conclude the paper with the following conjecture.

\begin{conjecture}
Let $G$ be a cyclically $4$-edge-connected cubic graph with a $2$-factor consisting of two circuits. Then the perfect matching index of $G$ is at most $4$ unless $G$ is the Petersen graph.
\end{conjecture}

\section*{Acknowledgements}
\noindent{}The authors of this article were partially supported
by the grant No.~APVV-23-0076 of the Slovak Research and
Development Agency. The first author was partially supported by the grant VEGA~2/0056/25 of the Slovak Ministry
of Education. The second author was partially 
supported by the grant VEGA~1/0173/25 of the Slovak Ministry of
Education.

\bibliographystyle{amsplain}
\bibliography{4cover}

\end{document}